\documentclass[12pt, 14paper,reqno]{amsart}
\usepackage{amsmath,amsfonts,amssymb}
\usepackage{mathrsfs}
\usepackage{longtable}
\usepackage[breaklinks]{hyperref}
\usepackage{graphicx}

\theoremstyle{plain}
\newtheorem{thm}{Theorem}[section]
\newtheorem{lem}{Lemma}[section]
\newtheorem{cor}{Corollary}[section]
\newtheorem{prop}{Proposition}[section]

\newtheorem{thma}{Theorem}

\theoremstyle{proof}

\numberwithin{equation}{section}

\begin{document} 
\title[On the Diophantine equation $cx^2+p^{2m}=4y^n$]{On the Diophantine equation $cx^2+p^{2m}=4y^n$}
\author{K. Chakraborty, A. Hoque and K. Srinivas}
\address{K. Chakraborty @Kerala School of Mathematics, Kozhikode-673571, Kerala, India.}
\email{kalychak@ksom.res.in}
\address{A. Hoque @Department of Mathematics, Rangapara College, Rangapara, Sonitpur-784505, India.}
\email{ahoque.ms@gmail.com}
\address{K. Srinivas @Institute of Mathematical Sciences, HBNI, CIT Campus, Taramani, Chennai-600113, India.}
\email{srini@imsc.res.in}
\date{\today}
\maketitle

\begin{abstract}
Let $c$ be a square-free positive integer and $p$ a prime satisfying $p\nmid c$. Let $h(-c)$ denote the class number of the imaginary quadratic field $\mathbb{Q}(\sqrt{-c})$. In this paper, we consider the Diophantine equation
$$cx^2+p^{2m}=4y^n,~~x,y\geq 1, m\geq 0, n\geq 3, \gcd(x,y)=1, \gcd(n,2h(-c))=1,$$ and we describe all its integer solutions. Our main tool here is the prominent result of Bilu, Hanrot and Voutier on existence of primitive divisors in Lehmer sequences.  
\end{abstract}

\section{Introduction}
Let $c$ and $d$ be positive integers such that $cd$ is square-free, and let $h(-cd)$ denote the class number of the imaginary quadratic field $\mathbb{Q}(\sqrt{-cd})$. There have been many papers concerned with the number of positive integer solutions $(x,n)$, for each fixed tuple $(c, d, p)$, of the generalized Ramanujan-Nagell equation 
\begin{equation}\label{eqi1}
cx^2+d=4p^n, ~~ x\geq 1, n\geq3.
\end{equation}
One of the finest known results is due to 
Bugeaud and Shorey, who proved in \cite{BS01} that \eqref{eqi1} has at most one solution for each given tuple $(c,d,p)$ with some exceptions. Many authors studied the following generalization of \eqref{eqi1}; so-called Lebesgue-Ramanujan-Nagell equation,
\begin{equation}\label{eqi2}
cx^2+d=4y^n,~~ x,y\geq 1, n\geq 3,\gcd(x,y)=1, \gcd(n, 2h(-cd))=1,
\end{equation}
where $x,y$ and $n$ are unknown integers. The known results include the following:
\begin{itemize}
\item (Persson \cite{PE49}, Stolt \cite{ST57}. If $c = 1$ and $n$ is a fixed odd prime, then there
exist only a finite number of $d$ for which \eqref{eqi2} has integer solutions $(x, y)$ and the number of solutions is also finite.

\item (Ljunggren \cite{LJ71, LJ72}). If $c = 1$ and $d$ satisfies some congruence conditions, then \eqref{eqi2} has only finite number of integer solutions $(x,y,n)$ with $n$ prime satisfying certain congruence conditions.

\item (Le \cite{LE93}).  If $n > 4\times 10^7$, then \eqref{eqi2} has no integer solution $(x, y)$.

\item (Luca et al. \cite{LTT09}). Determined all the integer solutions $(x,y,n)$ for $c=1$, $d\equiv 3\pmod 4$ and $1\leq d\leq 100$.  
\end{itemize}

In this paper, we consider more generalized Diophantine equation:
\begin{equation}\label{eqi3}
cx^2+d^m=4y^n,~~ x, y\geq 1, \gcd(cx,dy)=1, m\geq 0, n\geq 3, \gcd(n,h(-cd))=1, 
\end{equation}
where $x,y, m$ and $n$ are integers. For odd $m$, \eqref{eqi3} was well studied, and thus there are handful number of results, which include the following:
\begin{itemize}
\item (Le \cite{LE-95}). For odd $m$ and prime $n>5$, \eqref{eqi3} has only finite number of solutions $(x,y,m,n)$.  Moreover, the solutions satisfy $4y^n<\exp\exp 470$. Later, this bound was improved by Mignotte \cite{MI97}. 

\item (Bugeaud \cite{BU01}, Arif and Al-Ali \cite{AA02}).  For odd $m$ and prime $n\geq 5$, they independently, determined the solutions of \eqref{eqi3}. However, Bilu \cite{BI02} pointed out that there is a flaw in \cite{LE-95, MI97} and {\it a fortiori} Bugeaud's result in \cite{BU01}, and corrected that inaccuracy. 
\item(Bhatter et al. \cite{BHS19}, Chakraborty et al. \cite{CHS20}). For $c=1$, odd $m$ and for certain primes $d$, they completely solved \eqref{eqi3} without the coprimality conditions. 
\item (Dabrowski et al. \cite{DGS20}). For square-free integer $c> 3$, $d$ a power of an odd prime, $m$ even and $n$ an odd prime, they studied \eqref{eqi3} under the conditions `$2^{n-1}d^{m/2}\not\equiv \pm 1\pmod c$' and $\gcd(d,n)=1$. Under these conditions, they solved \eqref{eqi3} for $a\in \{7, 11, 19, 43, 67, 163\}$. Here, we completely solved  \eqref{eqi3} without these conditions and we allow $n\geq 3$ to be any integer. 

\end{itemize}

Our aim is to perform a deeper study on \eqref{eqi3} for even $m$ and prime $d$. More precisely, we completely solved the Diophantine equation 
\begin{equation}\label{eqn1}
cx^2+p^{2m}=4y^n,~~ x\geq 1, y> 1, m\geq 0, n\geq 3, \gcd(x,y)=1, \gcd(n, h(-c))=1,   
\end{equation}
in integers $x,y,m,n$ and prime $p$ with $p\nmid c$. Notably, we show that in some cases \eqref{eqn1} can be completely solved by appealing to a deep result of Bilu, Hanrot and Voutier \cite{BH01} on the existence of primitive divisors of Lehmer numbers.

\section{Statement of the results}
We first fix some notations. 
For a positive odd integer $t$, we define the following:
$$\mathcal{R}(c,u,v,t):=\sum_{j=0}^{\frac{t-1}{2}} \binom{t}{2j}u^{t-2j-1}c^{\frac{t-1}{2}-j}(-v^2)^j,$$
$$\mathcal{I}(c,u,v,t):=\sum_{j=0}^{\frac{t-1}{2}} \binom{t}{2j+1}u^{t-2j-1}c^{\frac{t-1}{2}-j}(-v^2)^j.$$
Then 
\begin{equation}\label{eqr}
\mathcal{R}(c,u,v,t)\equiv \begin{cases} u^{t-1}c^{(t-1)/2} \pmod t,\\
t (-v)^{(t-1)/2} \pmod c.
\end{cases}
\end{equation}
and 
\begin{equation}\label{eqi}
\mathcal{I}(c,u,v,t)\equiv \begin{cases} (-v)^{(t-1)/2} \pmod t,\\
(-v)^{(t-1)/2} \pmod c.
\end{cases}
\end{equation}

The main goal of this paper is to prove the following result.
\begin{thm}\label{thm}
Let $c>3$ be a square-free integer and $p$ an odd prime. 
Suppose that $n\geq 3$ is an integer such that $\gcd(n, 2h(-c))=1$. 
If $n$ has a prime factor $q$ satisfying $2^{q-1}p^m\ne|\mathcal{I}(c, u, 1, q)|$, then \eqref{eqn1} has no solutions $(c,x,y,m,n)$ with integers $x, y, m$ and $n$. 
Further, if all the prime factors $q$ of $n$ satisfy $2^{q-1}p^m=|\mathcal{I}(c, u, 1, q)|$, then for such a  prime factor $q$,  the following hold.
\begin{itemize}
\item[(I)] If $n=q$, then the solutions of \eqref{eqn1} are given by 
$$(x, y)=\left(\frac{u|\mathcal{R}(c, u,1,q)|}{2^{q-1}}, \frac{u^2c+1}{4}\right),$$
where $u\geq 1$ is an odd integer.


\item[(II)] If $n=rq$ for some odd integer $r\geq 3$,  
then \eqref{eqn1} has no solutions. 
\end{itemize}
\end{thm}

\subsection*{Remarks} Here, we make the following comments on Theorem \ref{thm}. 
\begin{itemize}
\item[(I)] For $c=1,2$, reading \eqref{eqn1} modulo $4$ we see that it has no solutions.  
\item[(II)] For $p=2$, \eqref{eqn1} can be written as $cX^2+2^{2m-2}=y^n$. This equation was deeply investigated by many authors; for instance,  Le \cite{LE95}, Muriefah \cite{MU01}, Wang and Wang \cite{YT11}, Luca and Soydan \cite{LS12} and Soydan and Cangul \cite{SC14}.
\item[(III)] The condition `$2^{q-1}p^m\not\equiv |\mathcal{I}(c, u, 1, q)|$' can be converted in to a stronger conditions `$p^m\not\equiv \pm 1\pmod q$' or `$2^{q-1}p^m\not\equiv \pm 1\pmod c$' using \eqref{eqi}.  
\end{itemize}

We now assume that $p$ is an odd prime and $n$ is an odd integer such that $p\mid n$. Then $p^m\ne \pm 1\pmod p$ for any integer $m\geq 1$. Also for $m=0$, \eqref{eqn1} becomes $cx^2+1=4(y^{n/p})^p$, which has no solutions by Lemma \ref{lemLJ}. 
Therefore, Theorem \ref{thm} (using (II) of Remarks) yields the following straightforward corollary.
\begin{cor}\label{cor1}
Let $c$ and $p$ be as in Theorem \ref{thm}. Assume that $n\geq 3$ is an odd integer such that $p\mid n$ and $\gcd(n, h(-c))=1$. The the Diophantine equation
$$cx^2+p^{2m}=4y^n,~~ x\geq 1, y>1, \gcd(x,y)=1,$$
has no solutions. In particular, the Diophantine equation 
$$cx^2+p^{2m}=4y^p,~~ x\geq 1, y>1, p\geq 3, \gcd(x,y)=1,$$
has no solutions.
\end{cor}
Assume that $\mathfrak{S}=\{7, 11,15,19,35,39,43,51,55,67,91,95,111,115,123,155,163,183, \\
187,195,203,219,235,259,267, 29,295,299,323,355,371,395,399,403,407,427,435,471,\\
483, 555,559,579,583,595,627,651,663,667,715,723,763,791,795,799,895,903,915,939,\\
943,955,979,987,995,1003,1015,1023,1027,1043,1047,1119,1131,1139,1155,1159,1195,\\
1227,1239,1243,1299,1339,1379,1387, 1411,1435,1443,1463,1507,1551,1555,1595,1635,\\
1651,1659,1731,1767,1771,1795,1803,1939,1943,1947,1983,1995
\}$. Then for any $c\in \mathfrak{S}$, $h(-c)\in\{1,2,4,8,16,32\}$. 
Thus one can deduce the following straightforward corollary.
\begin{cor}
Let $c\in\mathfrak{S}$, and $p\geq 3$ be an odd prime.
Then the Diophantine equation
$$cx^2+p^{2m}=4y^p,~~ x, y\geq 1, m\geq 0, \gcd(x,y)=1,$$ 
has no solutions.
\end{cor}

We now consider an integer $n\geq 3$. Assume that $q$ be an odd prime such that $q\mid n$. Then  for any prime $p$, we have $2^{q-1}p^q\equiv \pm 1\pmod q$, which implies $p\equiv \pm 1\pmod q$. Therefore by Theorem \ref{thm} (using (II) of Remarks), one gets the following:
\begin{cor}\label{cor2}
Let $n\geq 3$ be an odd integer and $q$ its prime factor such that $q\nmid h(-c)$. If $p$ is an odd prime such that $p\not\equiv \pm 1\pmod q$, then the Diophantine equation
$$qx^2+p^{2q}=4y^n, ~~x, y\geq 1, \gcd(x,y)=1,$$
has no solutions. 
\end{cor}
Note that it follows from Corollary \ref{cor2} that for distinct odd primes $p$ and $q$ satisfying $q\nmid p\pm 1$, the Diophantine equation
$$qx^2+p^{2q}=4y^q, ~~x, y\geq 1, \gcd(x,y)=1,$$
has no solutions, since $h(-q)<q$.

Theorem \ref{thm} yields the following corollary. 
\begin{cor}
Fix $p\in \{3,7,11,19,43,67,163\}$ and  $q \in \{19,43,67,163\}$. Then the Diophantine equation 
$$qx^2+p^{2m}=4y^q,~~x,y\geq1, m\geq 0, \gcd(x,y)=1,$$
has no solutions.
\end{cor}

\section{Preliminaries}
We first recall important results on the existence of primitive divisors of Lehmer numbers. For algebraic integers $\alpha$ and $\beta$, the pair $(\alpha, \beta)$ is known as a Lehmer pair if $(\alpha + \beta)^2$ and $\alpha\beta$ are two non-zero coprime rational integers, and $\alpha/\beta$ is not a root of unity. For a given positive integer $\ell$, the $\ell$-th Lehmer number corresponds to the pair $(\alpha, \beta)$ is defined as 
$$\mathcal{L}_\ell(\alpha, \beta)=\begin{cases}
\dfrac{\alpha^\ell-\beta^\ell}{\alpha-\beta} & \text{ if } 2\nmid \ell, \vspace{1mm}\\
\dfrac{\alpha^\ell-\beta^\ell}{\alpha^2-\beta^2} & \text{ if } 2\mid \ell.
\end{cases}$$
It is known that all Lehmer numbers are non-zero rational integers. 
Note that two Lehmer pairs $(\alpha_1, \beta_1)$ and $(\alpha_2, \beta_2)$ are equivalent if $\alpha_1/\alpha_2=\beta_1/\beta_2\in \{\pm 1, \pm\sqrt{-1} \}$. A prime divisor $p$ of $\mathcal{L}_\ell(\alpha, \beta)$ is primitive if $p\mid\mathcal{L}_\ell(\alpha, \beta)$ and $p\nmid(\alpha^2-\beta^2)^2
\mathcal{L}_1(\alpha, \beta) \mathcal{L}_2(\alpha, \beta) \cdots \mathcal{L}_{\ell-1}(\alpha, \beta)$. 

Let $a=(\alpha+\beta)^2$ and $b=(\alpha-\beta)^2$. Then $\alpha=(\sqrt{a}\pm\sqrt{b})/2$ and $\beta=(\sqrt{a}\mp\sqrt{b})/2$. The pair $(a, b)$ is known as the parameter corresponding to the Lehmer pair $(\alpha, \beta)$. 
We now recall the classical result of Bilu et al. \cite[Theorem 1.4]{BH01}. 
\begin{thma}\label{thmBH}
For any integer $\ell>30$, the Lehmer numbers $\mathcal{L}_\ell (\alpha, \beta) $ have primitive divisors.
\end{thma}

The following lemma is a pert of a classical result of  Voutier \cite[Theorem 1]{VO95}.
\begin{lem}\label{lemVO}
Let $\ell$ be a prime such that $7\leq\ell\leq 29$. Assume that the Lehmer numbers $\mathcal{L}_\ell(\alpha, \beta)$ have no primitive divisor, then up to equivalence, the parameters $(a, b)$ of the corresponding pair $(\alpha, \beta)$ must be given by the following:
\begin{itemize}
\item[(i)] For $\ell=7, (a, b)=(1,-7), (1, -19), (3, -5), (5, -7), (13, -3), (14, -22)$,
\item[(ii)] For $\ell=13, (a, b)=(1,-7)$.
\end{itemize}
\end{lem}
Let $F_k$ denote the $k$-th term in the Fibonacci sequence defined by $F_0=0,   F_1= 1$,
and $F_{k+2}=F_k+F_{k+1}$, where $k\geq 0$ is an integer. Also let $L_k$ denote the $k$-th term in the Lucas sequence defined by  $L_0=2,  L_1=1$, and $L_{k+2}=L_k+L_{k+1}$. 
We need the following lemma, which follows from \cite[Theorem 1.3]{BH01}. 
\begin{lem}\label{lemBH}
For $\ell=3, 5$, let the Lehmer numbers $\mathcal{L}_\ell(\alpha, \beta)$ have no primitive divisor. Then up to equivalence, the parameters $(a, b)$ of the corresponding pair $(\alpha, \beta)$ must be:
\begin{itemize}
\item[(i)] For $\ell=3, (a, b)=(1+t, 1-3t) \text{ with } t\ne 1, (3^k+t, 3^k-3t) \text{ with } t\not\equiv 0\pmod 3, (k,t)\ne(1,1)$;
\item[(ii)] For $\ell=5, (a, b)=(F_{k-2\varepsilon}, F_{k-2\varepsilon}-4F_k)\text{ with } k\geq 3,  (L_{k-2\varepsilon}, L_{k-2\varepsilon}-4L_k)\text{ with } k\ne 1$;
\end{itemize}
where $t\ne 0$ and $k\geq 0$ are any integers and $\varepsilon=\pm 1$.
\end{lem}

\begin{thma}[{\cite[Theorems 1 and 3]{CO64}}]\label{thmCO} For an integer $k\geq 0$, let $F_k$ (resp. $L_k$) denote the $k$-th Fibonacci (resp. Lucas) number. Then
\begin{itemize}
\item[(i)] if $L_k=x^2$, then $k=1,3$;
\item[(ii)] if $F_k=x^2$, then $k=0, 1,2,12$.
\end{itemize}
\end{thma}

We deduce the following lemma from \cite[Theorem 3.3]{KK15}.
\begin{lem}\label{lemKK}
Let $F_k$ be as in Theorem \ref{thmCO}. If $F_k=5x^2$, then $k=5$. 
\end{lem}

We also recall the following lemma from \cite[Lemma 2.1]{HO20}.
\begin{lem}\label{lemfl} Let $F_k$ and $L_k$ be as in Theorem \ref{thmCO}.
 Then for $\varepsilon=\pm 1$,
\begin{itemize}
\item[(i)] $4F_k-F_{k-2\varepsilon}=L_{k+\varepsilon}$,
\item[(ii)] $4L_k-L_{k-2\varepsilon}=5F_{k+\varepsilon}$
\end{itemize}
\end{lem}

We deduce the following lemma from  \cite[Corollary 3.1]{YU05}.
\begin{lem}\label{lemYU}
Let $C$ be a square-free positive integer. Assume that $t$ is a positive odd integer such that $\gcd(t, h(-C))=1$. Then all the positive integer solutions $(X,Y,Z)$ of the equation 
\begin{equation*}\label{eqyu}
CX^2+Y^2=4Z^n,~~ \gcd(CX, Y)=1,
\end{equation*} 
can be expressed as 
$$\frac{X\sqrt{C}+Y\sqrt{-1}}{2}=\delta\left(\frac{U\sqrt{C}+\mu V\sqrt{-1}}{2}\right)^t,$$
where $U$ and $V$ are positive integers satisfying $4Z=CU^2+V^2$, $\gcd(UC, V)=1$, and $\delta, \mu\in\{-1, 1\}$ when $C>3$. For $C=3$, $\delta\in\{-1, 1, (1+\sqrt{-C})/2, (-1+\sqrt{-C})/2, (1-\sqrt{-C})/2,(-1-\sqrt{-C})/2 \}$ and $\mu\{-1, 1\}$.
\end{lem}

We recall the following lemma from a result of Ljunggren \cite{LJ43}.
\begin{lem}\label{lemLJ} Let $C\geq 3$ be a square-free integer and $n\geq 3$ an integer such that $\gcd(n,2h(-C))=1$.  Then the Diophantine equation   
$$
CX^2+1=4Y^n, ~~X\geq 1, Y>1, \gcd(X,Y)=1,$$
has  no solutions.
\end{lem}

\section{Proof of Theorem \ref{thm}}
Here, we first prove the following crucial proposition which is the main ingredient in the proof of Theorem \ref{thm}.
\begin{prop}\label{propq}
Let $c>3$ be a square-free integer. Assume that $p$ and $q$ are odd primes such that $q\nmid h(-c)$. Then the Diophantine equation
\begin{equation}\label{eqq1}
cx^2+p^{2m}=4y^q, ~~\gcd(x,y)=1,
\end{equation}
has no positive integer solution, except for 
$p^m=|\mathcal{I}(u,1,q)|/2^{q-1}$. Further for $p^m=|\mathcal{I}(u,1,q)|/2^{q-1}$, its possible solutions are given by 
$$(x, y)=\left(\frac{u|\mathcal{R}(u,1,q)|}{2^{q-1}}, \frac{u^2c+1}{4}\right),$$
where $u\geq 1$ is an odd integer.
\end{prop}
\begin{proof}[{\bf{Proof of Proposition \ref{propq}}}] 
Assume that $(c,p,x,y,m, n)$ is a positive integer solution of \eqref{eqq1}. Then $x$ is odd, and thus $c\equiv 3\pmod 4$. Also note that $p\nmid cx$ since $\gcd(x,y)=1$ and $c$ is square-free. As $q$ is odd prime and $q\nmid h(-c)$, so that by Lemma \ref{lemYU} we can positive integer $u$ and $v$ such that 
\begin{equation}\label{eqq2}
\frac{x\sqrt{c}+p^m\sqrt{-1}}{2}=\delta\left(\frac{u\sqrt{c}+\mu v\sqrt{-1}}{2}\right)^q
\end{equation}
 and 
 \begin{equation}\label{eqq3}
 4y=u^2c+v^2
 \end{equation}
 with $\gcd(uc,v)=1$, where $\delta, \mu\in\{-1, 1\}$.

Equation the real and imaginary parts in \eqref{eqq3}, we get
\begin{align}\label{eqq4}
\begin{cases}
x=\dfrac{\delta u}{2^{q-1}}\mathcal{R}(u,v,q),\vspace{2mm}\\
p^m=\dfrac{\delta \mu v}{2^{q-1}}\mathcal{I}(u,v,q).
\end{cases}
\end{align}
These imply both $u$ and $v$ are odd, since $x$ and $p$ are odd.

We now define,
$$\alpha=\frac{u\sqrt{c}+\mu v \sqrt{-1}}{2},~~ \beta=\frac{u\sqrt{c}-\mu v\sqrt{-1}}{2}.$$
Then one can show, using \eqref{eqq3}, that both $\alpha$ and $\beta$ are algebraic integers. Clearly, $\gcd(uc, y)=1$, and thus $(\alpha+\beta)^2=u^2c$ and $\alpha\beta=y$ are coprime. Also $$\frac{\alpha}{\beta}=\frac{(u^2c-v^2)/2+\mu u v\sqrt{-c}}{y}$$ is a root of $$yT^2-(u^2c-v^2)T+4y=0.$$ 
This shows that $\alpha/\beta$ is not a root of unity, since $\gcd(u^2c-v^2, y)=\gcd(v,y)=1$. Thus $(\alpha, \beta)$ is a Lehmer pair with the parameter $(u^2c, -v^2)$.  

Let $\mathcal{L}_t(\alpha,\beta)$ be the Lehmer number for the pair $\alpha, \beta$. Then $\mathcal{L}_q(\alpha, \beta)=p^m/\delta\mu v$. Since $\mathcal{L}_q(\alpha,\beta)$ is a non-zero rational integer and $p$ is an odd prime, we get the following three cases.
 
\subsection*{Case I} $v=p^m$. 
In this case, $|\mathcal{L}_q(\alpha, \beta)|=1$, and hence $\mathcal{L}_q(\alpha, \beta)$ has no primitive divisors. Therefore by Theorem \ref{thmBH} and Lemmas \ref{lemVO} and \ref{lemBH}, there is no Lehmer number
$\mathcal{L}_q(\alpha, \beta)$ which has no primitive divisors when $q\ne 3,5,7,13$. Since $(u^2c, v^2)$ is the parameter corresponds to the Lehmer pair $(\alpha, \beta)$, so that by Lemma \ref{lemVO}, $q=7,13$ are not possible. This ensures that \eqref{eqq1} has no positive integer solution, except for $q=3,5$. 

For $q=3$, by Lemma \ref{lemBH} (since $(u^2c, -v^2)$ is the parameter corresponds to $(\alpha, \beta)$), 
\begin{equation}\label{eqq31}
v^2= 3t-1
\end{equation}
and 
\begin{equation}\label{eqq32}
v^2=3t-3^k,
\end{equation}
where $t\ne 0$ and $k\geq 1$ are integers with $t\not\equiv 0\pmod 3$ and $(t, k)\ne (1,1)$.   Clearly, \eqref{eqq31} is not possible. Since $t\not\equiv 0\pmod 3$, so that  \eqref{eqq32} gives $k=1$ and thus $v\equiv 0\pmod 3$. This ensures that $v=3^m$, and thus equating the imaginary parts in \eqref{eqq2} for $q=3$, we get
$3u^2c-3^{2m}=\pm 4$. This is not possible.    

If $q=5$, then we get (by Lemma \ref{lemBH}),
\begin{equation}\label{eqq51}
(u^2c,v^2)=(F_{k-2\varepsilon}, 4F_k-F_{k-2\varepsilon}),~~k\geq 3, 
\end{equation}
and 
\begin{equation}\label{eqq52}
(u^2c,v^2)=(L_{k-2\varepsilon}, 4L_k-L_{k-2\varepsilon}), ~~k\ne 1,
\end{equation}
where $\varepsilon=\pm1$. Utilizing Lemma \ref{lemfl} and Theorem \ref{thmCO} in \eqref{eqq51}, we get $(k, \varepsilon)=(4,-1)$. Thus, $u^2c=F_7=13$ which gives $u=1$ and $c=13$. This is not possible since $c\equiv 3\pmod 4$. 
Again utilizing Lemma \ref{lemfl} in \eqref{eqq52}, we get $5F_{k+\varepsilon}=v^2$, which shows that $v=5^m$. Thus $F_{k+\varepsilon}=5\times 5^{2m-2}$, and hence by Lemma \ref{lemKK}, $(k, \varepsilon, m)\in\{(4,1,1), (6,-1,1)\}$. Therefore $u^2c=L_{k-2\varepsilon}$ gives $u^2c=3,47$, which further implies $(u, c)\in\{(1,3), (1,47)\}$. These lead to $3x^2+5^2=4\times 2^5$ and $47x^2+5^2=4\times 13^5$. It is easy to check that these are not possible.  

\subsection*{Case II} $v=p^r, 1\leq r\leq m-1$.  In this case, we get $|\mathcal{L}_q(\alpha, \beta)|=p^{m-r}$. Also $(\alpha^2-\beta^2)^2=-cu^2p^{2r}$. Since $r\ne 0$, so that $\mathcal{L}_q(\alpha, \beta)$ has no primitive divisors. We can follow case I to show that \eqref{eqq1} has no solution in this case too.  

\subsection*{Case III} $v=1$. In the case, $|\mathcal{L}_q(\alpha, \beta)|=p^m.$ Thus $p$ is a primitive divisor of $\mathcal{L}_q(\alpha, \beta)$, and hence the previous method is not utilizable. Therefore, we utilize \eqref{eqq4} to 
$$x=\frac{|u\mathcal{R}(c,u,1,q)|}{2^{q-1}}$$
and 
$$p^m=\frac{|\mathcal{I}(c,u,1,q)|}{2^{q-1}}.$$ 
Also \eqref{eqq3} gives $y=(u^2c+1)/4$. 
\end{proof}

\begin{proof}[{\bf Proof of Theorem \ref{thm}}]
Assume that $(c,x,y,m,n)$ is a positive integer solution of \eqref{eqn1}. Then as in Proposition \ref{propq}, $x$ is odd, $c\equiv 3\pmod 4$ and $p\nmid cx$. 

Suppose that $q$ is a prime factor of $n$. Then we can write \eqref{eqn1} as follows:
\begin{equation}\label{eqn2}
cx^2+p^{2m}=4Y^q, 
\end{equation}
where $Y=y^{n/q}$. Thus by Proposition \ref{propq}, \eqref{eqn2} has no solutions in positive integers provided $p^m\ne|\mathcal{I}(c,u,1,q)/2^{q-1}$. Therefore \eqref{eqn1} has no solutions in positive integers if $p^m\ne|\mathcal{I}(c,u,1,q)/2^{q-1}$. 

In rest of the proof, we consider $p^m=\mathcal{I}(c,u,1,q)/2^{q-1}$. Therefore by Proposition \ref{propq},  the only solution of \eqref{eqn2} is given by 
$$(x, Y)=\left(\frac{u|\mathcal{R}(c,u,1,n)|}{2^{n-1}}, \frac{u^2c+1}{4}\right).$$
In case of $n=q$, we can conclude that the only solution of \eqref{eqn1} is as follows:
$$(x, y)=\left(\frac{u|\mathcal{R}(c,u,1,n)|}{2^{n-1}}, \frac{u^2c+1}{4}\right),$$
where $u\geq 1$ is an odd integer.

Further, if $n=rq$ for some odd integer $r\geq 3$, then by Proposition \ref{propq}, we get 
\begin{equation}\label{eqn3}
cu^2+1=4y^r, ~~u\geq 1, y>1, \gcd(u,y)=1.
\end{equation}
Since $\gcd(r, h(-c))=1$, so that by Lemma \ref{lemLJ}, \eqref{eqn3} has no solution $(c,u,y,r)$ in positive integers. Therefore, \eqref{eqn1} has no solutions. This completes the proof. 
\end{proof}

\section*{acknowledgements}
This work was started when K. Chakraborty and A. Hoque were visited Professor Jianya Liu at Shandong University, Weihai, whereas it was completed while A. Hoque was visiting K. Srinivas at The Institute of Mathematical Sciences, Chennai. K. Chakraborty and A. Hoque are thankful to Professor Jianya Liu and Shandong University, Weihai for wonderful hospitality and support during their visiting period.  A. Hoque would like to thank The Institute of Mathematical Sciences, Chennai for hospitality and financial support during his visiting period. The authors are grateful to the anonymous referee for careful reading and valuable comments which have helped to improve this paper. The authors are also grateful to the referee for drawing the paper \cite{DGS20} to their attention.
This work is supported by the grants SERB MATRICS Project No. MTR/2017/001006 and SERB-NPDF (PDF/2017/001958), Govt. of India.

\end{document}